\documentclass{amsproc}

\usepackage{amsthm,amssymb,amscd,graphicx,mathrsfs,pinlabel} 

\newtheorem{theorem}{Theorem}
\newtheorem{lemma}[theorem]{Lemma}
\newtheorem{proposition}[theorem]{Proposition}

\theoremstyle{definition}
\newtheorem{definition}[theorem]{Definition}

\theoremstyle{remark}
\newtheorem{remark}[theorem]{Remark}

\newcommand{\co}{\colon \thinspace}
\newcommand{\Q}{{\mathbb Q}}
\newcommand{\R}{{\mathbb R}}
\newcommand{\Z}{{\mathbb Z}}
\newcommand{\T}{{\mathscr T}}
\newcommand{\abs}[1]{\left| {#1} \right|}
\renewcommand{\leq}{\leqslant}
\renewcommand{\geq}{\geqslant}
\DeclareMathOperator{\scl}{scl}

\begin{document}

\title[Turn Graphs and Extremal Surfaces in Free Groups]{Turn Graphs and Extremal Surfaces in Free Groups} 


\author{Noel Brady}
\address{Mathematics Department, University of Oklahoma\\ Norman, OK
73019, USA}
\email{nbrady@math.ou.edu}

\author{Matt Clay}
\address{Mathematics Department, Allegheny College\\ Meadville, PA 16335,
USA} 
\email{mclay@allegheny.edu}

\author{Max Forester}
\address{Mathematics Department, University of Oklahoma\\ Norman, OK
73019, USA}
\email{forester@math.ou.edu}

\thanks{Partially supported by NSF grants DMS-0906962 (Brady),
DMS-1006898 (Clay), and DMS-0605137 (Forester).} 

\date{}

\begin{abstract}
This note provides an alternate account of Calegari's rationality theorem
for stable commutator length in free groups. 
\end{abstract}

\maketitle


\section{Introduction}

The purpose of this note is to provide an alternate account of
Calegari's main result from \cite{Calegari:free}, establishing the
existence of extremal surfaces for stable commutator length in free
groups, via linear programming. The argument presented here is similar to
that given in \cite{Calegari:free}, except that we avoid using the theory
of branched surfaces. Instead, the reduction to linear programming is
achieved directly, using the combinatorics of words in the free group. 
We note that the specific linear programming problem resulting from the 
discussion here essentially agrees with that described in Example 4.34 of
\cite{Calegari:scl}. 

\subsubsection*{Acknowledgements} The authors would like to thank Dan
Guralnik and Sang Rae Lee for helpful discussions during the course
of this work. 

\section{Preliminaries} 

We start by giving a working definition of stable commutator
length. Propositions 2.10 and 2.13 of \cite{Calegari:scl} show that it is
equivalent to the basic definition in terms of commutators or genus. 

\begin{definition}\label{def1} 
Let $G = \pi_1(X)$ and suppose $\gamma \co S^1 \to X$ represents the
conjugacy class of $a\in G$. The \emph{stable commutator length of $a$}
is given by 
\begin{equation}\label{scldef} 
 \scl(a) \ = \ \inf_S \frac{-\chi(S)}{2n(S)} 
\end{equation}
where $S$ ranges over all singular surfaces $S \to X$ such that 
\begin{itemize}
\item $S$ is oriented and compact with $\partial S \not= \emptyset$ 
\item $S$ has no $S^2$ or $D^2$ components 
\item the restriction $\partial S \to X$ factors through $\gamma$; that
is, there is a commutative diagram: 
\[\begin{CD}
\partial S @>>> S \\
@VVV @VVV\\
S^1 @>{\gamma}>> X
\end{CD} \]
\item the restriction of the map $\partial S \to S^1$ to each connected
component of $\partial S$ is a map of \emph{positive} degree 
\end{itemize}
and where $n(S)$ is the total degree of the map $\partial S \to S^1$ (of
oriented $1$--manifolds). 
\end{definition}

A surface $S$ satisfying the conditions above is called a \emph{monotone
admissible surface} in \cite{Calegari:scl}, abbreviated here as an
\emph{admissible surface}. Such a surface exists if and only if $a \in 
[G,G]$. If $a \not\in [G,G]$ then by convention $\scl(a) = \infty$ (the
infimum of the empty set). 

A surface $S \to X$ is said to be \emph{extremal} if it realizes the
infimum in \eqref{scldef}. Notice that if this occurs, then $\scl(a)$ is a
rational number.

\section{Singular surfaces in graphs} \label{surfaces} 

Let $X$ be a graph with oriented $1$--cells $\{e_{\nu}\}$. These edges
may be formally considered as a generating set for the fundamental
groupoid of $X$ based at the vertices. These generators also generate the
fundamental group $F$ of $X$. Note that $F$ is free, but the groupoid
generators are not a basis unless $X$ has only one vertex. (The reader
may assume this latter property with no harm, in which case the
fundamental groupoid is simply the fundamental group.) 

Let $\gamma \co S^1 \to X$ be a simplicial loop with no
backtracking. There is a corresponding cyclically reduced word $w = x_1
\cdots x_{\ell}$ in the fundamental groupoid generators and their
inverses. This word $w$ represents a conjugacy class in $\pi_1(X)$, which
we assume to be in $[F,F]$. Finally, let $S \to X$ be an admissible
surface for $w$, as in Definition \ref{def1}. 

We are interested in computing $\chi(S)$ and $n(S)$, to estimate
$\scl(w)$ from above. We are free to modify $S$ if the resulting surface
$S'$ satisfies $\frac{-\chi(S')}{2n(S')} \ \leq
\ \frac{-\chi(S)}{2n(S)}$, since this only strengthens the estimate. 

Using transversality, the map $S \to X$ can be homotoped into a standard
form, sometimes called a transverse map \cite{BRS}. The surface is
decomposed into pieces called $1$--handles, which map to edges of $X$,
and complementary regions, which map to vertices of $X$. Each $1$--handle
is a tubular neighborhood of a connected $1$--dimensional submanifold, 
either an arc with endpoints on $\partial S$ or a circle. The 
submanifold maps to the midpoint of an edge of $X$, and the fibers of
the tubular neighborhood map over the edge, through its 
characteristic map. In particular, the boundary arcs or
circles of the $1$--handle (comprised of endpoints of fibers) map
to vertices of $X$. A \emph{transverse labeling} is a labeling of the
fibers of $1$--handles by fundamental groupoid generators, indicating
which edge of $X$ (and in which direction) the handle maps to. 
For more detail on putting maps into this form, see
for instance \cite{Rourke,Stallings,Culler,BF}. 

Let $S_0 \subset S$ be the codimension-zero submanifold obtained as the
closure of the union of a collar neighborhood of $\partial S$ and the
$1$--handles that meet $\partial S$. We will see that $S_0$ is the essential
part of $S$, containing all of the relevant information. It is 
determined completely by $\partial S$, together with the additional
data of which pairs of edges in $\partial S$ are joined by $1$--handles.
Note that $\partial S_0$ consists of $\partial S$ together with
additional components in the interior of $S$. These latter components
will be called the \emph{inner boundary} of $S_0$, denoted
$\partial^-S_0$. Let $S_1$ be the closure of $S - S_0$. Note that
$\partial S_1 = S_1 \cap S_0 = \partial^- S_0$. Figure \ref{fig:S0} shows
an example of $S_0$ for the word $w= a b a^{-1}b^{-1}$. (The ``turns''
mentioned there are discussed in the next section.) 
\begin{figure}[ht]
  \centering
\labellist
\hair 3pt
\pinlabel* {{\scriptsize $1$}} at 31 18
\pinlabel* {{\scriptsize $2$}} at 55 18
\pinlabel* {{\scriptsize $3$}} at 79 18
\pinlabel* {{\scriptsize $4$}} at 103 18
\pinlabel* {{\scriptsize $1$}} at 127 18
\pinlabel* {{\scriptsize $2$}} at 151 18
\pinlabel* {{\scriptsize $3$}} at 175 18
\pinlabel* {{\scriptsize $4$}} at 199 18
\pinlabel* {{\scriptsize $1$}} at 223 18
\pinlabel* {{\scriptsize $2$}} at 247 18
\pinlabel* {{\scriptsize $3$}} at 271 18
\pinlabel* {{\scriptsize $4$}} at 295 18
\pinlabel* $a$ [B] at 269 67
\pinlabel* $b$ [B] at 293 67
\pinlabel {$\partial^- S_0$} [l] at 300 25
\pinlabel {$\partial S$} [l] at 300 11
\pinlabel* {$w^3 \longrightarrow$} [l] at 2 2
\pinlabel {$\alpha$} [tl] at 80 78
\endlabellist
  \includegraphics{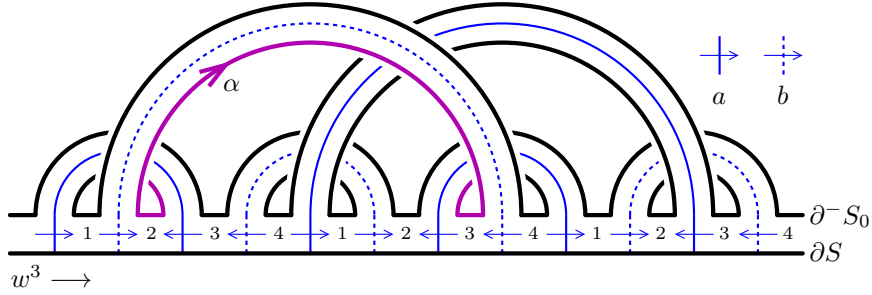} 
  \caption{One possible $S_0$ for the word $w = aba^{-1}b^{-1}$, with
outer boundary $w^3$. There are four turns, indicated by the numbers,
each occurring three times. The boundary arc $\alpha$ leads from turn $2$
to turn $3$. This surface leads to an estimate $\scl(w) \leq 1/2$ (and in
fact is extremal).} 
  \label{fig:S0}
\end{figure}

How large can $\chi(S)$ be? Note that $\chi(S) = \chi(S_0) + \chi(S_1)$
since $S_0$ and $S_1$ meet along circles. Also, 
\[\chi(S_0) \ = \ \frac{-n(S) \abs{w}}{2}\] 
as can be seen by counting the $1$--handles meeting $\partial
S$: each $1$--handle contributes $-1$ to $\chi(S_0)$ and occupies two
edges in $\partial S$, of which there are $n(S)\abs{w}$ in total. 
Finally, given $S_0$, the quantity $\chi(S_1)$ is largest when $S_1$ is a
collection of disks. The number of disks is simply the number of
components of $\partial^- S_0$. We can always replace $S_1$ by disks,
since each component of $\partial S_1$ maps to a vertex of $X$ and disks
can be mapped to vertices also. Thus, after this modification, we have 
\begin{equation} \label{scleqn} 
\chi(S) \ = \ \frac{-n(S) \abs{w}}{2} \ + \ \abs{\pi_0   (\partial^-
S_0)},
\end{equation}
and therefore an upper bound for $\scl(w)$ is given by 
\begin{equation}\label{scldef2}
\frac{-\chi(S)}{2n(S)} \ = \ \frac{\abs{w}}{4} \ - \
\frac{\abs{\pi_0(\partial^- S_0)}}{2 n(S)} .
\end{equation}
Indeed, $\scl(w)$ is precisely the infimum of the right hand side of 
\eqref{scldef2} over all surfaces $S_0$ arising as above. (Note that
$n(S)$ is determined by $S_0$.) Equation \eqref{scldef2} essentially
replaces the quantity $\chi(S)$ by the number of inner boundary
components of $S_0$ in the computation of $\scl(w)$.

\section{The turn graph} 

To help keep track of the inner boundary $\partial^- S_0$, we define 
the \emph{turn graph}. Consider the word $w = x_1 \cdots x_{\ell}$. 
A \emph{turn} in $w$ is a position between two letters of $w$ considered
as a cyclic word. 
Turns are indexed by the numbers $1$ through $\ell$, with turn $i$ being
the position just after the letter $x_i$. 
Each turn is
labeled by the length two subword $x_i x_{i+1}$ (or $x_{\ell}x_1$) of $w$
which straddles the turn. Note that turns are not necessarily determined
by their labels. 

The \emph{turn graph} $\Gamma(w)$ is a directed graph with vertices equal
to the turns of $w$, and with a directed edge from turn $i$ to turn $j$
if $x_{i}^{-1} = x_{j+1}$. That is, if the label of a turn begins with
the letter $a^{\pm 1}$, then there is a directed edge from this turn to
every other turn whose label ends with $a^{\mp 1}$. Note that because $w$
is cylically reduced, $\Gamma(w)$ has no loops.  

The turn graph has a two-fold symmetry, or duality: if $e \in \Gamma(w)$
is an edge from turn $i$ to turn $j$, then one verifies easily that there
is also an edge $\bar{e}$ from turn $j+1$ to turn $i-1$, and moreover
$\bar{\bar{e}} = e$. Figure \ref{fig:turngraph_2} shows a turn graph and
a dual edge pair. 
\begin{figure}[ht]
  \centering
\labellist
\hair 1pt
\pinlabel* $ab$ at 111 188
\pinlabel* $ba$ at 155 176
\pinlabel* $ab$ at 187 144
\pinlabel* $b\bar{a}$ at 198 101
\pinlabel* $\bar{a}\bar{b}$ at 187 57
\pinlabel* $\bar{b}a$ at 155 25
\pinlabel* $a\bar{b}$ at 111 13
\pinlabel* $\bar{b}\bar{a}$ at 67 25
\pinlabel* $\bar{a}b$ at 35 57
\pinlabel* $b\bar{a}$ at 23 101
\pinlabel* $\bar{a}\bar{b}$ at 35 144
\pinlabel* $\bar{b}a$ at 67 176
\pinlabel {\scriptsize $1$} [l] at 126 188
\pinlabel {\scriptsize $2$} [l] at 170 176
\pinlabel {\scriptsize $3$} [l] at 202 144
\pinlabel {\scriptsize $4$} [l] at 213 101
\pinlabel {\scriptsize $5$} [l] at 202 57
\pinlabel {\scriptsize $6$} [l] at 170 25
\pinlabel {\scriptsize $7$} [l] at 126 13
\pinlabel {\scriptsize $8$} [r] at 52 25
\pinlabel {\scriptsize $9$} [r] at 20 57
\pinlabel {\scriptsize $10$} [r] at 8 101
\pinlabel {\scriptsize $11$} [r] at 20 144
\pinlabel {\scriptsize $12$} [r] at 52 176
\endlabellist
  \includegraphics{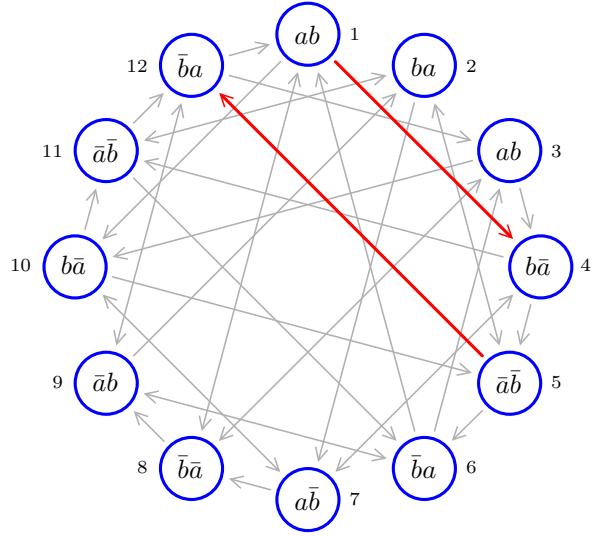} 
  \caption{The turn graph for the word $w =
    abab\bar{a}\bar{b}a\bar{b}\bar{a}b\bar{a}\bar{b}$ (bar denotes
inverse). The highlighted edges form a dual pair.} 
  \label{fig:turngraph_2}
\end{figure}

\subsubsection*{Turn circuits} 
Given the surface $S_0$, each inner boundary component can be described
as follows. Traversing the curve in the positively oriented direction,
one alternately follows $1$--handles and visits turns of $w$
positioned along $\partial S$; see again Figure \ref{fig:S0} (this
situation is the reason for the word ``turn''). If a $1$--handle leads
from turn $i$ to turn $j$, then the $1$--handle bears the transverse
label $x_i^{-1} = x_{j+1}$, and so there is an edge in $\Gamma(w)$ from
turn $i$ to turn $j$. The sequence of $1$--handles traversed by the
boundary component therefore yields a directed circuit in $\Gamma(w)$. In
this way the inner boundary $\partial^- S_0$ gives rise to a finite
collection (possibly with repetitions) of directed circuits in
$\Gamma(w)$, called the \emph{turn circuits} for $S_0$. 

Recall that $\partial S$ is labeled by $w^{n(S)}$ (possibly spread over
several components), so there are $n(S)$ occurrences of each turn on
$\partial S$. The turn circuits do not contain the information of which
particular instances of turns are joined by $1$--handles.

\subsubsection*{Turn surgery} There is a move one can perform on $S_0$
which is useful. Given two occurrences of turn $i$ in $S_0$, cut the
collar neighborhood of $\partial S$ open along arcs positioned at the two
turns, between the adjacent $1$--handles; see Figure \ref{fig:move}. 
\begin{figure}[ht]
  \centering
\labellist
\hair 3pt
\pinlabel {$x_i$} [Br] at 34 101
\pinlabel {$x_i$} [Br] at 34 57
\pinlabel {$x_{i+1}$} [tl] at 138 65
\pinlabel {$x_{i+1}$} [tl] at 138 21
\pinlabel {$\partial^- S_0$} [Br] at 19 121
\pinlabel {$\partial S$} [Br] at 10 93
\pinlabel {$\partial S$} [Br] at 10 49
\endlabellist
  \includegraphics{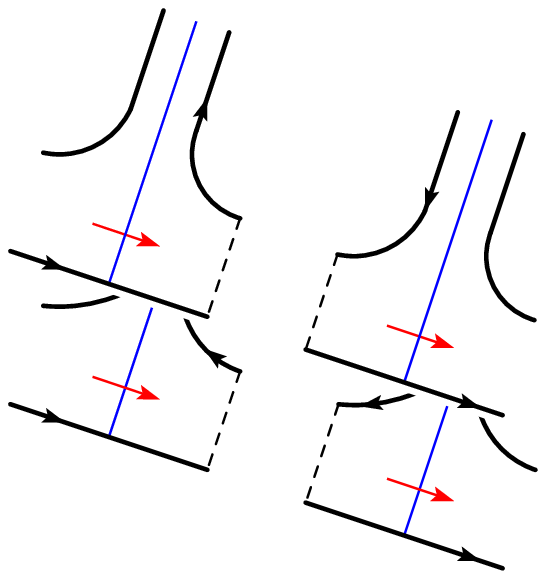} 
  \caption{Cutting along two instances of turn $i$.}
  \label{fig:move}
\end{figure}
These arcs both map to the same vertex of $X$. Now re-glue the four sides
of the arcs, switching two of them. There is one way to do this which
preserves orientations of $S$ and of $\partial S_0$. The new surface is
still admissible (that is, after capping off $\partial^- S_0$) and $n(S)$
is preserved. 

The move changes both $\partial S$ and $\partial^- S_0$, in each
case either increasing or decreasing the number of connected components
by one. If both instances of the turn occupy the same component, then the
move splits this component into two, with each occupied by one of the
turns. Otherwise, the move joins the two components occupied by the turns
into one. 

\begin{definition} An admissible surface $S$ is \emph{taut} if every
component of $\partial^- S_0$ visits each turn at most once. In terms of
the turn graph, this means that each turn circuit for $S_0$ is embedded
in $\Gamma(w)$ (though distinct circuits are allowed to cross). Let
$\T(w)$ be the set of taut admissible surfaces for $w$. 
\end{definition}

Any admissible surface $S$ can be made taut by performing a finite number
of turn surgeries, each increasing the number of inner boundary
components of $\partial S_0$. Since $n(S)$ remains constant, the quantity
\eqref{scldef2} will only decrease. Hence we have the following result: 

\begin{lemma} \label{tautlemma} 
There is an equality 
\[{\displaystyle \scl(w) \ \ = \ \ \inf_{S \in \T(w)}
\ \frac{\abs{w}}{4} \ - \ \frac{\abs{\pi_0(\partial^- S_0)}}{2 n(S)} .}\]  
\end{lemma}

\section{Weight vectors and linear optimization}

Let $\{\alpha_1, \ldots, \alpha_k\}$ be the set of embedded directed
circuits in $\Gamma(w)$. For each taut admissible surface $S$ let
$u_i(S)$ be the number of occurrences of $\alpha_i$ among the turn
circuits of $S_0$, and let $u(S) \in \R^k$ be the non-negative integer
vector $(u_1(S), \ldots, u_k(S))$. We call $u(S)$ the \emph{weight
vector} for $S$. 

For each vertex $v$ and edge $e$ of $\Gamma(w)$, there are linear
functions
\[ F_v\co \R^k \to \R, \quad F_e \co \R^k \to \R\]
whose values on the $i^{th}$ standard basis vector $(0, \ldots, 0, 1, 0,
\ldots 0)$ are given by the number of times $\alpha_i$ passes through the
vertex $v$ (respectively, over the edge $e$). Since $\alpha_i$ is
embedded, these numbers will be $0$ or $1$, although this is not
important. For the taut surface $S$, if $e \in \Gamma(w)$ is an edge from
turn $i$ to turn $j$, then $F_e(u(S))$ counts the number of times
$\partial^- S_0$ follows a $1$--handle from turn $i$ to turn
$j$. Similarly, if $v \in \Gamma(w)$ is turn $i$, then $F_v(u(S))$ counts
the number of occurrences of turn $i$ on $\partial^- S_0$ (which is
$n(S)$, as observed earlier). 

\begin{remark} \label{remark} 
For taut admissible surfaces, the functions $\abs{\pi_0(\partial^- S_0)}$
and $n(S)$ both factor as 
\[  \T(w) \overset{u}{\longrightarrow} \R^k \longrightarrow \R \]
where the second map is linear, with integer coefficients. In the case of
$\abs{\pi_0(\partial^-   S_0)}$ the second map is given by $(u_1, \ldots,
u_k) \mapsto \sum_iu_i$, and in the case of $n(S)$, the second map is
simply $F_v$ (for any vertex $v \in \Gamma(w)$). 

By \eqref{scleqn} it follows that the function $-\chi(S)$ also factors as
above, through an integer coefficient linear function $\R^k \to \R$. 
\end{remark}

\begin{lemma}\label{eqnlemma} 
Every weight vector $u(S)$ satisfies the linear equation 
\[F_e(u(S)) = F_{\bar{e}}(u(S))\]
for every dual pair $e, \bar{e}$ of edges in $\Gamma(w)$. 
\end{lemma}

\begin{proof}
Suppose $e$ leads from turn $i$ to
turn $j$ (so $\bar{e}$ leads from turn $j+1$ to turn $i-1$). If a
$1$--handle has a boundary arc representing $e$ then the other side of
the $1$--handle represents $\bar{e}$. Hence both sides of the equation
count the number of $1$--handles of $S_0$ joining occurrences of $x_i$
and $x_{j+1}$ in $\partial S$. 
\end{proof}

This lemma has a converse: 

\begin{proposition} \label{eqnprop} 
If $u \in \R^k - \{0\}$ has non-negative integer entries and satisfies the
linear equations 
\begin{equation} \label{subspace} 
F_e(u) = F_{\bar{e}}(u) \ \ \text{ for all dual pairs } e, \bar{e} 
\end{equation}
then $u$ is the weight vector of a taut admissible surface. 
\end{proposition}

\begin{proof}
Suppose $u = (u_1, \ldots ,u_k)$. For each $i$ let $D_i$ be a polygonal
disk with $\abs{\alpha_i}$ sides. Label the oriented boundary 
of $D_i$ by the edges and vertices of $\alpha_i$. That is, sides are
labeled by edges of $\Gamma(w)$, and corners are labeled by turns. 
Note that there are no monogons, since $\Gamma(w)$ has no
loops. To form the taut 
admissible surface $S$, take $u_i$ copies of $D_i$ for each $i$. For each
dual edge pair $e$, $\bar{e}$ the total number of edges labeled $e$
among the $D_i$'s will equal the number of edges labeled $\bar{e}$,
by \eqref{subspace}. Hence the sides of the disks can be
joined in dual pairs to form a closed oriented surface. 

However, this is not how $S$ is formed. Instead, whenever two disks were
to be joined along sides labeled $e$ and $\bar{e}$, insert an
oriented rectangle, with sides labeled by $e$, $x_{j+1}$, $\bar{e}$,
$x_i$ (here, $e$ leads from turn $i$ to turn $j$, and $\bar{e}$ from turn
$j+1$ to turn $i-1$). See Figure \ref{fig:realization}. The opposite
sides labeled by 
$e$ and $\bar{e}$ are joined to the appropriate sides of the disks, and
the remaining two sides become part of the boundary of $S$. Each
rectangle can be transversely labeled by a fundamental groupoid generator
(equal to $x_i^{-1} = x_{j+1}$), and then the rectangles become
$1$--handles in the resulting surface $S$. 

\begin{figure}[ht]
  \centering
\labellist
\pinlabel {$\partial S$} [r] at 36 65
\pinlabel {$e$} [r] at 140 71
\pinlabel {$\bar{e}$} [l] at 193 71
\pinlabel {$x_{j+1}$} [t] at 166 22
\pinlabel {$x_{i}$} [b] at 166 122
\pinlabel* {$x_{i+1}$} at 138 143
\pinlabel {$\partial S$} [r] at 155 3
\pinlabel {$\partial S$} [l] at 172 143
\pinlabel* {$i$} at 136 104
\pinlabel* {$i-1$} at 205 114
\pinlabel* {$j+1$} at 205 32
\pinlabel* {$j$} at 136 29
\pinlabel* {$i+1$} at 95 157
\endlabellist
  \includegraphics{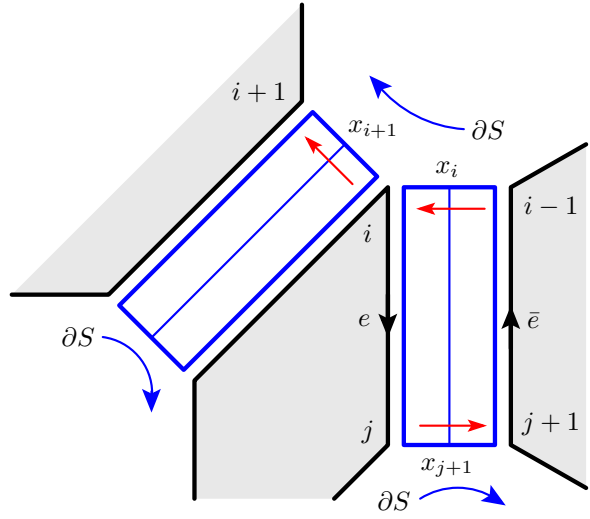} 
  \caption{Building $S$ from $(u_1, \ldots, u_k)$.}
  \label{fig:realization}
\end{figure}

Note that the side of a rectangle labeled $x_{i}$ has neighboring
polygonal disk corners labeled $i-1$ and $i$. Following this edge 
along $\partial S$, the next edge must be labeled $x_{i+1}$ (adjacent to
$i$ and $i+1$); see again Figure \ref{fig:realization}. Hence each
component of  
$\partial S$ is labeled by a positive power of $w$. There are no $S^2$
components since no component of $S$ is closed, and no $D^2$ components, 
since an outermost $1$--handle on such a disk would have to bound a
monogon. The map $S \to X$ is defined on
the rectangles according to the transverse labels (each maps to an edge
of $X$) and the disks map to vertices. Now $S$ is admissible, and by
construction, the turn circuits will all be instances of the circuits
$\alpha_i$, so $S$ is taut. 
\end{proof}

\begin{theorem}[Calegari] 
If $X$ is a graph and $a \in [\pi_1(X), \pi_1(X)]$ then there exists an
extremal surface $S \to X$ for $a$. Moreover, there is an algorithm to 
construct $S$. In particular, $\scl(a)$ is rational and computable. 
\end{theorem}

\begin{proof} Let $w$ be the cyclically reduced word representing the
conjugacy class of $a$, as defined in Section \ref{surfaces}. 
By Remark \ref{remark} the function $\frac{-\chi(S)}{2n(S)}$ factors as 
\[\T(w) \overset{u}{\longrightarrow} \R^k {\longrightarrow} \R\]
where the second map is a ratio of linear functions $A(u)/B(u)$ with
integer coefficients. Let $P \subset \R^k$ be the polyhedron defined by
the (integer coefficient) linear equations \eqref{subspace} and the
inequalities $B(u) > 0$ and $u_i \geq 0$, $i = 1, \ldots, k$. Lemma
\ref{eqnlemma} and Proposition \ref{eqnprop} together imply that 
the image of $u \co \T(w) \to \R^k$ is precisely $P \cap \Z^k$. Hence 
\begin{equation*}
\scl(w) \ \ = \ \ \inf_{u \in P\cap \Z^k} \ {A(u)}/{B(u)}. 
\end{equation*}
Note that $P$ and $A(u)/B(u)$ are projectively invariant. Normalizing
$B(u)$ to be $1$, we have  
\begin{equation}\label{scldef3} 
\scl(w) \ \ = \ \ \inf_{u \in P' \cap \Q^k} \ A(u)
\end{equation} 
where $P'$ is the rational polyhedron $P \cap B^{-1}(1)$. Note that $P'$
is a closed set. 

From Remark \ref{remark} and equation \eqref{scleqn} the function $A$ is
given by 
\[ A(u) \ \ = \ \ \frac{\abs{w}}{2} F_v(u) + \sum_i u_i\, ,  \]
which has strictly positive values on the standard basis vectors. 
Hence $A$ achieves a minimum on $P'$, along a non-empty rational
sub-polyhedron. The vertices of this sub-polyhedron are rational points
realizing the infimum in \eqref{scldef3}. Hence there exist extremal
surfaces for $w$, and $\scl(w)$ is rational. An extremal surface 
can be constructed explicitly from a rational solution $u \in P'\cap
\Q^k$, by first multiplying by an integer to obtain a minimizer for $A(u)/B(u)$
in $P \cap \Z^k$, and then applying the procedure given in the proof of
Proposition \ref{eqnprop}. Lastly, we note that from the word $w$ it is
straightforward to algorithmically construct the turn graph $\Gamma(w)$,
the equations \eqref{subspace}, and the polyhedron $P'$. 
\end{proof}

\bibliographystyle{amsplain}

\providecommand{\bysame}{\leavevmode\hbox to3em{\hrulefill}\thinspace}
\providecommand{\MR}{\relax\ifhmode\unskip\space\fi MR }
\providecommand{\MRhref}[2]{%
  \href{http://www.ams.org/mathscinet-getitem?mr=#1}{#2}
}
\providecommand{\href}[2]{#2}
\begin{thebibliography}{1}

\bibitem{BF}
Noel Brady and Max Forester, \emph{Density of isoperimetric spectra}, Geom.
  Topol. \textbf{14} (2010), no.~1, 435--472. \MR{2578308}

\bibitem{BRS}
S.~Buoncristiano, C.~P. Rourke, and B.~J. Sanderson, \emph{A geometric approach
  to homology theory}, Cambridge University Press, Cambridge, 1976, London
  Mathematical Society Lecture Note Series, No. 18. \MR{0413113 (54 \#1234)}

\bibitem{Calegari:scl}
Danny Calegari, \emph{scl}, MSJ Memoirs, vol.~20, Mathematical Society of
  Japan, Tokyo, 2009. \MR{2527432}

\bibitem{Calegari:free}
\bysame, \emph{Stable commutator length is rational in free groups}, J. Amer.
  Math. Soc. \textbf{22} (2009), no.~4, 941--961. \MR{2525776}

\bibitem{Culler}
Marc Culler, \emph{Using surfaces to solve equations in free groups}, Topology
  \textbf{20} (1981), no.~2, 133--145. \MR{605653 (82c:20052)}

\bibitem{Rourke}
C.~P. Rourke, \emph{Presentations and the trivial group}, Topology of
  low-dimensional manifolds ({P}roc. {S}econd {S}ussex {C}onf., {C}helwood
  {G}ate, 1977), Lecture Notes in Math., vol. 722, Springer, Berlin, 1979,
  pp.~134--143. \MR{547460 (81a:57001)}

\bibitem{Stallings}
John~R. Stallings, \emph{A graph-theoretic lemma and group-embeddings},
  Combinatorial group theory and topology ({A}lta, {U}tah, 1984), Ann. of Math.
  Stud., vol. 111, Princeton Univ. Press, Princeton, NJ, 1987, pp.~145--155.
  \MR{895613 (88k:20056)}

\end{thebibliography}

\providecommand{\bysame}{\leavevmode\hbox to3em{\hrulefill}\thinspace}
\providecommand{\MR}{\relax\ifhmode\unskip\space\fi MR }
\providecommand{\MRhref}[2]{%
  \href{http://www.ams.org/mathscinet-getitem?mr=#1}{#2}
}
\providecommand{\href}[2]{#2}

\end{document}